\documentclass[a4paper,numbook,babel,final,envcountsame,11pt]{article}

\usepackage{amssymb}
\usepackage{amsthm}
\usepackage{amsmath}
\usepackage{graphicx}
\usepackage{array,hhline}

\newtheorem{lemma}{Lemma}[section]
\newtheorem{theorem}[lemma]{Theorem}

\theoremstyle{definition}

\numberwithin{equation}{section}
\numberwithin{figure}{section}

\begin{document}
\title{\huge On the factorization of linear combinations of polynomials}         
\author{Anna Gharibyan}        
\date{}          

\maketitle

\begin{abstract}
In this paper we consider linear combinations of two trivariate homogeneous polynomials of second degree. We formulate and solve two problems:

i) Characterization of polynomials for which all linear combinations are factorizable.

ii) How many linear factorizable combinations are required for all linear combinations to be factorizable.

Next, the solutions of analog problems for bivariate polynomials of second degree are derived.

\end{abstract}
%
%

\section{Introduction}

Among the properties of multivariate polynomials the factorizability and irreducibility are very important (see \cite{Ga} -\cite{Wa}).

The factorizability reduces the study of a polynomial into the study of two or more polynomials of smaller degrees.

While many properties are true only for irreducible polynomials. For example, let $p$ and $q$
be bivariate polynomials of degree $m$ and $n,$ respectively  where $m>n,$ and the polynomial $p$ be irreducible. Then, in view of the Bezout theorem, the polynomial system $p(x,y)=0, q(x,y)=0,$ has atmost $mn$ solutions.

In this paper we consider linear combinations of two trivariate homogeneous polynomials of degree $2$.
We study the following problem: Is there a linear combination which is irreducible, or alternatively, are all linear combinations factorizable? We bring a simple necessary and sufficient condition for the latter property.

We also consider the following problem: How many linear factorizable combinations are required for all linear combinations to be factorizable. We determine the exact number of such linear combinations - $4.$ We bring also counterexample in which case only $3$ linear combinations are factorizable.

Next, we establish the analogs of the mentioned results for bivarite polynomials of degree two.

Let us use the following notation for the space of trivariate homogeneous polynomials:
$$\dot\Pi_n^3=\{p(x,y,z)= \sum_{i+j+k=n}a_{ijk}x^iy^jz^{k}\},$$
and the space of bivariate polynomials:
$$\Pi_n^2=\{p(x,y)= \sum_{i+j\le n}a_{ij}x^iy^j\},$$
with complex coefficients.

We use the following well-known concept of the associate polynomial (see section 10.2, \cite{Wal}) .
Let $p\in \Pi_n^2,$ i.e., $p(x,y)=\sum_{i+j\le n}a_{ij}x^iy^j.$ Then the following trivariate homogeneous polynomial is called associated with $p:$
$$\bar p(x,y,z)= \sum_{i+j+k=n}a_{ij}x^iy^jz^{k}.$$
Evidently we have that
\begin{equation}\label{aaa}\bar p(x,y,z)=z^np\left(\frac{x}{z},\frac{y}{z}\right),\ \hbox{for}\ z\neq 0. \end{equation}
Also we have that $$p(x,y)=\bar p(x,y,1).$$
It is easily seen that a polynomial $p\in\Pi_n^2$ is factorizable if and only if the associated polynomial $\bar p$ is factorizable. Moreover, we have that
$$p=p_1p_2 \Leftrightarrow \bar p=\bar p_1\bar p_2.$$

To simplify notation, we shall use the same letter $\ell$, say, to denote the bivariate polynomial of degree $1$ and the line described by the equation $\ell(x,y)=0,$ or the trivariate homogeneous polynomial of degree $1$ and the line in the projective space described by the equation $\ell(x,y,z)=0.$

The following result follows from the fundamental theorem of algebra (see Theorem 10.8, \cite{Wal}).
\begin{theorem}\label{hom_n} Let $p$ be a bivariate homogeneous polynomial of degree $n:\ p\in\dot\Pi_n^2.$ Then $p$ can be factorized into linear polynomials:
$$p=\prod_{i=1}^n\ell_i,$$
where $\ell_i\in\dot\Pi_1^2.$
\end{theorem}
	\section{The results for homogeneous polynomials}
\begin{theorem} \label{tm:1}
	Let $p(x, y, z)$ and $q(x, y, z)$ be linearly independent second degree homogeneous polynomials. Then the polynomial $r(x, y, z) = \alpha p(x, y, z)$ + $\beta q(x, y, z)$ is a product of polynomials of first degree for all $\alpha$, $\beta$ $\in \mathbb C$ if and only if at least one of the following conditions takes place:
	\begin{description}
		\item[$a)$] $p$ and $q$ have a common factor of first degree:
$$p=\ell \ell_1,\ q=\ell \ell_2,\ \hbox{where}\ \ell,\ell_1,\ell_2\in \dot\Pi_1^3$$
		\item[$b)$] $p(x, y, z) = f(s, t), \ q(x, y, z) = g(s, t),$	
		 where $f$ and $g$ are bivariate homogeneous polynomials of degree two: $f,g\in\dot\Pi_2^2$ and $s,t\in\dot\Pi_1^3:$\\
		\begin{equation} \label{st} s = a_1x + b_1y + c_1z,\
		 t = a_2x + b_2y + c_2z.
\end{equation}
	\end{description}
\end{theorem}

	Let us mention that the statement in ``if" direction of Theorem \ref{tm:1} is evident. Indeed, if a) takes place, then
	$$\alpha p+\beta q = \ell(\alpha \ell_{1}+\beta \ell_2).$$
			
	If b) takes place then $f,g\in\dot\Pi_2^2$ and therefore $r=\alpha p+\beta q=\alpha f+\beta g\in\dot\Pi_2^2$ is a bivariate homogeneous polynomial of second degree of $s$ and $t$. Therefore, in view of Theorem \ref{hom_n}, $r(x, y, z)$ can be factorized in the variables of $s$ and $t$ for all $\alpha, \beta\in\mathbb C.$
Then by using \eqref{st} and passing to the variables $x,y,z,$ we get the desired factorization. \newline
	Now let us prove the necessity of the conditions a) and b). It is enough to prove the following theorem.\\
\begin{theorem} \label{tm:2}	
	Let $p(x, y, z)$ and $q(x, y, z)$ be linearly independent homogeneous polynomials of degree two. Suppose there exist four noncollinear vectors $(\alpha_i, \beta_i) \in \mathbb C^2,\ i = 1,\dots, 4,$  such that the polynomials
	$r_i(x, y, z) = \alpha_i p(x, y, z) + \beta_i q(x, y, z)$
	can be factorized.

Then we have that at least one of the conditions a) and b) of Theorem \ref{tm:1} takes place and therefore, for all $\alpha, \beta\in \mathbb C,$ the polynomilal
	$$r(x, y, z) = \alpha p(x, y, z) + \beta q(x, y, z)$$
	is factorizable.\\
\end{theorem}	
\begin{proof} The proof consists of two parts.

{\bf Step 1.}
	Let us first consider the case $(\alpha_1, \beta_1)=(1,0)\quad (\alpha_2, \beta_2)=(0,1).$ Then we have that  $(\alpha_3, \beta_3)$ and $(\alpha_4, \beta_4)$ are not collinear and
	\begin{equation} \label{nez}\alpha_3 \beta_3 \alpha_4 \beta_4\neq 0.
\end{equation}
Now let us prove that at least one of the conditions a) and b) of Theorem \ref{tm:1} takes place.

To this end let us assume that the condition a) does not hold, i.e., $p$ and $q$ have no common factor. Then we shall prove that the condition b) holds.

We have that $p$ and $q$ can be factorized:
	\begin{equation}\label{pqell} p(x, y, z) =\ell_1 \ell_2,\quad q(x, y, z) = \ell_3 \ell_4,\end{equation}
where the line $\ell_i$ is given by $A_i x + B_i y + C_iz=0.$

	It suffices to prove that the lines $\ell_1, \ell_2,\ell_3, \ell_4$ are concurrent, i.e., \\ $rank\{\ell_1, \ell_2,\ell_3, \ell_4\} = 2$.
Indeed, in this case, by setting, say $s= A_1 x + B_1 y + C_1z=:\ell_1$ and $t= A_3 x + B_3 y + C_3z=:\ell_3$ we get for $\ell_2$ and $\ell_4:$  $\ A_2 x + B_2 y + C_2z=\alpha_2 t+\beta_2s$ and $\ A_4 x + B_4 y + C_4z=\alpha_4 t+\beta_4s.$
Hence, in view of \eqref{pqell}, we get $p=t(\alpha_2 t+\beta_2s)$ and $q=s(\alpha_4 t+\beta_4s).$ Thus the condition b) of Theorem \ref{tm:1} holds.

	Suppose by way of contradiction  that $rank\{\ell_1, \ell_2,\ell_3, \ell_4\} = 3$. It means that there are 3 lines which are not concurrent. Without loss of generality assume that these lines are $\ell_1, \ell_2,\ell_3.$ Let us set:
	$$ \ell_1 = A_1 x + B_1 y + C_1z =: \bar{x}, $$
	$$ \ell_2 = A_2 x + B_2 y + C_2z =: \bar{y}, $$
	$$ \ell_3 = A_3 x + B_3 y + C_3z =: \bar{z}. $$
Then we have that
$$ x = A_1' \bar x +  B_1'\bar y + C_1'\bar z, $$
	\begin{equation}\label{ABC} y = A_2'\bar x + B_2'\bar y + C_2'\bar z, \end{equation}
	$$ z = A_3'\bar x + B_3'\bar y + C_3'\bar z. $$

	We have that $r_3(x, y, z)=\alpha_3\ell_1\ell_2 + \beta_3\ell_3\ell_4$ can be factorized. Therefore, in view of \eqref{ABC}, there are numbers $A_5, B_5, C_5, A_6, B_6, C_6\in \mathbb C,$ depending on  $\alpha_3$ and $\beta_3$ such that
	\begin{equation} \label{eq:1}
		\alpha_3\bar{x}\bar{y} + \beta_3\bar{z}(A_4\bar{x} + B_4\bar{y} + C_4\bar{z}) = (A_5\bar{x} + B_5\bar{y} + C_5\bar{z})(A_6\bar{x} + B_6\bar{y} + C_6\bar{z}).
	\end{equation}
	
	In the left part of $\eqref{eq:1}$ the coefficients of $\bar{x}^2$ and $\bar{y}^2$ equal to 0, therefore
	$$A_5 A_6 = 0,$$
	$$B_5 B_6 = 0.$$
	Without loss of generality we can discuss only two cases: $A_5 = 0, B_5 = 0$ or $A_5 = 0, B_6 = 0.$

		First consider the case $A_5 = B_5 = 0.$ From \eqref{ABC} we obtain
		\begin{equation*}
		\alpha_3\bar{x}\bar{y} + \beta_3\bar{z}(A_4\bar{x} + B_4\bar{y} + C_4\bar{z}) = C_5\bar{z}(A_6\bar{x} + B_6\bar{y} + C_6\bar{z}).
		\end{equation*}
		Here the coefficient of $\bar{x}\bar{y}$ in the right hand side equals to 0. Therefore $\alpha_3 = 0,$ which contradicts \eqref{nez}.
		
		Now consider the case $A_5 = B_6 = 0.$ From $\eqref{eq:1}$ we get:
		\begin{equation} \label{eq:2}
			\alpha_{3}\bar{x}\bar{y} + \beta_{3}\bar{z}(A_4\bar{x} + B_4\bar{y} + C_4\bar{z}) = (B_5\bar{y} + C_5\bar{z})(A_6\bar{x} + C_6\bar{z}).
		\end{equation}
		
First, let us verify that $A_4 \neq 0.$ Assume by way of contradiction that $A_4=0.$ By comparing the coefficients of $\bar x\bar z$ in both sides of \eqref{eq:2} we get $A_6C_5=0.$ If $A_6 = 0,$ then the coefficient of $\bar x\bar y$ in the left hand side of \eqref{eq:2} equals to $0,$ i.e., $\alpha_3 = 0,$ which contradicts \eqref{nez}. If $C_5 = 0$, then the coefficient of $\bar z^2$ in the left hand side of \eqref{eq:2} equals to $0,$ i.e.,  $\beta_3 C_4 = 0.$ In view of \eqref{nez}, $\beta_3\not = 0$, hence $C_4=0.$ But in this case $\bar{y}$ becomes a common factor of $p$ and $q,$ which contradicts our assumption.

Now let us set:
$$\beta_3' = \beta_3 A_4,\quad B_4' = \frac{B_4}{A_4},\quad C_4' = \frac{C_4}{A_4}.$$
Then we get from \eqref{eq:2}:
\begin{equation} \label{eq:3}	\alpha_3\bar{x}\bar{y} + \beta_3'\bar{z}(\bar{x} + B_4'\bar{y} + C_4'\bar{z}) = (B_5\bar{y} + C_5\bar{z})(A_6\bar{x} + C_6\bar{z}).
\end{equation}
			
				Next, let us verify that $B_4'C_4' = 0.$  Assume by way of contradiction that $B_4'C_4' \not = 0.$ Then by compairing the coefficients in both sides of \eqref{eq:3} we obtain:
				\begin{equation} \label{eq:6}
				      \alpha_3 = A_6B_5,\quad
					\beta_3' = A_6C_5,\quad
					\beta_3' B_4' = B_5C_6,\quad
					\beta_3' C_4' = C_5C_6.
								\end{equation}
				Notice that $A_6C_5\neq 0,$ since, in view of \eqref{nez}, $\beta_3\neq 0,$ and hence $\beta_3'\neq 0.$  Then, notice that $C_6\neq 0.$ Indeed if $C_6=0$ then, in view of \eqref{eq:6}, we obtain that $B_4'=C_4'=0.$ But then $p$ and $q$ have common factor $\bar x,$ which is a contradiction.

Now, from first two equations of \eqref{eq:6} we get that
				$$\frac{B_5}{C_5} = \frac{\alpha_3}{\beta_3'}.$$	
				From last two equations of \eqref{eq:6} we get:
				$$\frac{B_5}{C_5} = \frac{B_4'}{C_4'}.$$	
				Therefore we have $$\frac{\alpha_3}{\beta_3'} = \frac{B_4'}{C_4'}.$$
				And from this we conclude that
				$$\frac{\alpha_3}{\beta_3} = \frac{A_4B_4}{C_4}.$$
				By the same way, by considering the linear combination $r_4(x,y,z)$ instead of $r_3(x,y,z)$ we will get:
				$$\frac{\alpha_4}{\beta_4} = \frac{A_4B_4}{C_4}.$$
				From the last two equalities we get that $(\alpha_3, \beta_3)$ and $(\alpha_4, \beta_4)$ are collinear, which contradicts the hypothesis.
				
		Thus we have that $B_4'C_4' = 0.$		Now, if $B_4' = 0,$ then we get from \eqref{eq:3}:
				\begin{equation} \label{eq:4}
				\alpha_3\bar{x}\bar{y} + \beta_3'\bar{z}(\bar{x} + C_4'\bar{z}) = (B_5\bar{y} + C_5\bar{z})(A_6\bar{x} + C_6\bar{z}).
				\end{equation}
				By comparing the coefficient of $\bar{y}\bar{z}$ in the left and right hand sides of  $\eqref{eq:4}$, we get that $B_5 C_6 = 0$. Now if $B_5 = 0,$ then the coefficient of $\bar x\bar y$ in the left hand side of \eqref{eq:4} equals to $0,$ i.e., $\alpha_3 = 0$, which contradicts \eqref{nez}. Then, if $C_6 = 0,$ the coefficient of $\bar z^2$ in the left hand side of \eqref{eq:4} equals to $0,$ i.e., $\beta_3' C_4' = 0.$ Since $\beta_3' \not = 0$, we conclude that $C_4'=0.$ But in this case $\bar{y}$ becomes a common factor of $p$ and $q,$ which is a contradiction.

			Next, if $C_4' = 0,$ we get from \eqref{eq:3}:	
				\begin{equation} \label{eq:5}
				\alpha_3\bar{x}\bar{y} + \beta_3'\bar{z}(\bar{x} + B_4'\bar{y}) = (B_5\bar{y} + C_5\bar{z})(A_6\bar{x} + C_6\bar{z}).
				\end{equation}
				By comparing the coefficient of $\bar{z}^2$ in the left and right hand sides of  $\eqref{eq:5}$, we get that $C_5 C_6 = 0$. Now, if $C_5 = 0$, then the coefficient of $\bar x\bar z$ in the left hand side of \eqref{eq:5} equals to $0,$ i.e., $\beta_3' = 0,$ which is a contradiction. If $C_6 = 0$, then the coefficient of $\bar y\bar z$ in the left hand side of \eqref{eq:5} equals to $0,$ i.e., $\beta_3' B_4' = 0$, therefore $B_4' = 0.$ But in this case $\bar{y}$ becomes a common factor of $p$ and $q,$ which is a contradiction.

{\bf Step 2.} Now let us turn to the general case.

Consider the polynomials
\begin{equation} \label{eq:9}
	\begin{aligned}
		\tilde{p} &= \alpha_1 p + \beta_1 q, \\
		\tilde{q} &= \alpha_2 p + \beta_2 q.
	\end{aligned}	
	\end{equation}
	We have that $\tilde{p}$ and $\tilde{q}$ are factorizable.
	Since $(\alpha_1, \beta_1)$ and $(\alpha_2, \beta_2)$ are noncollinear we obtain:
	\begin{equation} \label{eq:10}
	\begin{aligned}
		p = \alpha_1'\tilde{p} + \beta_1' \tilde{q},\\
		q = \alpha_2'\tilde{p} + \beta_2' \tilde{q},
	\end{aligned}
	\end{equation}
	where
	\begin{equation} \label{eq:11}
		\begin{pmatrix}
			\alpha_1 & \beta_1 \\
			\alpha_2 & \beta_2
		\end{pmatrix}
		\begin{pmatrix}
			\alpha_1' & \beta_1' \\
			\alpha_2' & \beta_2'
		\end{pmatrix}
		=
		\begin{pmatrix}
			1 & 0 \\
			0 & 1
		\end{pmatrix}.
	\end{equation}
Hence, we get that $(\alpha_1', \beta_1')$ and $(\alpha_2', \beta_2')$ are noncollinear.

	Then we have that the following polynomials
	$$r_3(x, y, z) = \alpha_3 p(x, y, z) + \beta_3 q(x, y, z),$$
	$$r_4(x, y, z) = \alpha_4 p(x, y, z) + \beta_4 q(x, y, z),$$
	can be factorized. By using the relations \eqref{eq:10} we get
	\begin{equation}\label{r3} r_3 = \alpha_3 (\alpha_1' \tilde p + \beta_1' \tilde q) + \beta_3 (\alpha_2' \tilde p + \beta_2' \tilde q) = \tilde{\alpha_3} \tilde p + \tilde{\beta_3} \tilde q,\end{equation}
	$$r_4 = \alpha_4 (\alpha_1' \tilde p + \beta_1' \tilde q) + \beta_4 (\alpha_2' \tilde p + \beta_2' \tilde q) = \tilde{\alpha_4} \tilde p + \tilde{\beta_4} \tilde q,$$
	where
	\begin{equation}  \label{matrix}
		\begin{pmatrix}
			\tilde\alpha_3 & \tilde\beta_3 \\
			\tilde\alpha_4 & \tilde\beta_4
		\end{pmatrix}
		=
		\begin{pmatrix}
		\alpha_3 & \beta_3 \\
		\alpha_4 & \beta_4
		\end{pmatrix}
		\begin{pmatrix}
		\alpha_1' & \beta_1' \\
		\alpha_2' & \beta_2'
		\end{pmatrix}.	
	\end{equation}
Since the vectors $({\alpha_3}, {\beta_3})$ and $({\alpha_4}, {\beta_4})$ as well as $(\alpha_1', \beta_1')$ and $(\alpha_2', \beta_2')$ are noncollinear, we get from \eqref{matrix} that the vectors  $(\tilde{\alpha_3}, \tilde{\beta_3}), (\tilde{\alpha_4}, \tilde{\beta_4})$ are noncollinear, too.

Now we are in the following position:
	
\noindent The polynomials $ \tilde p(x, y, z)$ and $ \tilde q(x, y, z)$ are factorizable themselves and  there exist two non-collinear vectors $(\tilde{\alpha_3}, \tilde{\beta_3}), (\tilde{\alpha_4}, \tilde{\beta_4})$  such that the polynomials
	$$r_3(x, y, z) = \tilde{\alpha_3} \tilde p(x, y, z) + \tilde{\beta_3} \tilde q(x, y, z),$$
	$$r_4(x, y, z) = \tilde{\alpha_4} \tilde p(x, y, z) + \tilde{\beta_4} \tilde q(x, y, z),$$
	 can be factorized.

	To be in the position of Step 1 it remains to show that
	$$\tilde{\alpha_3} \tilde{\beta_3} \tilde{\alpha_4} \tilde{\beta_4} \not = 0.
$$
Suppose by way of contradiction that, for example $\tilde{\alpha_3} = 0$. According to \eqref{r3} we have that
	$$\alpha_3 \alpha_1'+ \beta_3 \alpha_2'= 0.$$
	On the other hand  we have from \eqref{eq:11} that
	$$\alpha_2 \alpha_1'+ \beta_2 \alpha_2'= 0.$$
	Therefore both $(\alpha_2, \beta_2)$ and $(\alpha_3, \beta_3)$ are orthogonal to $(\alpha_1', \alpha_2')$ and hence are collinear, which contradicts to the hypothesis.

Therefore, in view of Step 1, we conclude that either the condition a) or the condition b) is satisfied with the polynomials $\tilde p$ and $\tilde q.$ Now, by taking into account the relation  \eqref{eq:10}, we readily get that at least one of conditions a) and b) of Theorem \ref{tm:1} holds. 
\end{proof}

Notice that the following result follows from the Step 1 of above proof:
\begin{theorem} \label{tm:4}
	Let $p= \ell_1 \ell_2, q = \ell_3\ell_4,$ where $\ell_i\in \dot\Pi_1^3,$ are linearly independent polynomials of degree two.  If the polynomial $r(x, y,z) = \alpha p(x, y,z)$ + $\beta q(x, y,z)$ is factorizable for two non-collinear vectors $(\alpha_1, \beta_1)$ and $(\alpha_2, \beta_2),$ with $\alpha_1\beta_1\alpha_2\beta_2\neq 0,$ then the polynomial $r(x,y,z)$ is factorizable for all pairs $(\alpha, \beta)\in \mathbb C^2.$

Moreover, at least one of the following conditions takes place:
	\begin{enumerate}
		\item[$(i)$] $p$ and $q$ have a common factor,
		\item[$(ii)$] The lines $\ell_1, \ell_2 , \ell_3, \ell_4$ are coincident.
	\end{enumerate}
\end{theorem}
\subsection {A counterexample\label{ce}}

Consider the following two polynomials:
$$p(x,y,z)=x(x+z),\quad  q(x,y,z)=y(2x+y+z).$$

Let us show that for the three non-collinear vectors $(1,0); \ (0,1);$ and $(1,1)$ the respective linear combinations are factorizable, while all other linear combinations are irreducible. Indeed, the case of first two linear combinations $(1,0)$ and $(0,1);$ is obvious.
Consider the third one: $(1,1).$ We have that

$$ p(x,y,z)+ q(x,y,z)=x(x+z)+y(2x+y+z)=x^2+xz+2xy+y^2+yz=$$
$$(x+y)^2+z(x+y)=(x+y)(x+y+z).$$

Now let us show that there is a forth linear combination which is not factorizable. Of course this, in view of Theorem \ref{tm:2}, will show that all linear combinations except the first three are irreducible.

Thus let us show that the linear combination corresponding to the vector $(2,1),$ i.e.,
   $$ 2p(x,y,z)+ q(x,y,z)=2x(x+z)+y(2x+y+z)$$
is irreducible. Assume, by way of contradiction that
$$ 2x^2+2xy+2xz+yz+y^2=A(x+B_1y+C_1z) (x+B_2y+C_2z).$$
Since the coefficients of $x^2$ and $z^2$ in the right hand side equal to $2$ and $0,$ respectively, we obtain that $A$=2 and $C_1C_2=0.$ Without loss of generality assume that $C_1=0.$ Thus we have that
$$ 2x^2+2xy+2xz+yz+y^2=2(x+B_1y) (x+B_2y+C_2z).$$
Now, by comparing the coefficients of $xy, xz, yz$ and $y^2,$ we obtain
$$B_1+B_2=1,\quad C_2=1,\quad 2B_1C_2=1,\quad 2B_1B_2=1.$$
Therefore we have from first three equalities: $B_1=B_2=0.5,$ which contradicts to the fourth equality: $2\cdot 0.5\cdot 0.5 =0.5\neq 1.$

\section{The results for bivariate polynomials}
	
	Now from the obtained results for trivariate homogeneous polynomials we are going to derive the analog results for bivariate polynomials:
\begin{theorem} \label{tm:20}	
	Let $p(x, y)$ and $q(x, y)$ be linearly independent polynomials of degree two. Suppose there exist four noncollinear vectors $(\alpha_i, \beta_i) \in \mathbb C^2,\ i = 1,\dots, 4,$  such that the polynomials
	$r_i(x, y) = \alpha_i p(x, y) + \beta_i q(x, y)$
	can be factorized.

Then we have that at least one of the following conditions takes place
\begin{description}
		\item[$a')$] $p$ and $q$ have a common factor of first degree:

$p=\ell \ell_1, q=\ell \ell_2$, where $\ell,\ell_1,\ell_2\in \dot\Pi_1^2$
		\item[$b')$] $p(x, y) = f(s), \ q(x, y) = g(s),$\\	
		 where $f$ and $g$ are univariate polynomials of degree two: $f,g\in\dot\Pi_2^1$ and \\
		$s = a_1x + b_1y,\ a_1^2+b_1^2\neq 0$,
		\item [$c')$] $p(x, y) = f(s, t), \ q(x, y) = g(s, t),$\\	
		 where $f$ and $g$ are bivariate homogeneous polynomials of degree two: $f,g\in\dot\Pi_2^2$ and $s = a_1x + b_1y + c_1,\
		 t = a_2x + b_2y + c_2.$
\end{description}
Moreover,  the polynomial
	$r(x, y) = \alpha p(x, y) + \beta q(x, y)$
	is factorizable for all $\alpha, \beta\in \mathbb C.$
\end{theorem}	
\begin{proof}
	Let  $\bar{p}$, $\bar{q}$ be the homogeneous polynomials associated with $p$ and $q.$ Then we have that
	$$p(x, y) = \bar p(x, y, 1),\quad q(x, y) = \bar q(x, y, 1).$$
We also have that $\bar r(x,y,z) =\alpha \bar p(x,y,z)+\beta\bar q(x,y,z)$ is associated with  $r(x, y) = \alpha p(x, y) + \beta q(x, y).$

Next notice that the hypothesis of Theorem \ref{tm:2} hold with the homogeneous polynomials $\bar{p}$, $\bar{q}.$ Thus we get that for any numbers $\alpha$ and $\beta$  the polynomial $\bar r(x,y,z)$ is factorizable and hence $r(x,y)=\bar r(x,y,1)$ is factorizable. It  remains to prove that at least one of conditions $a'),b'),c')$ holds.

Suppose first that the condition a) of Theorem \ref{tm:1} holds, i.e.,
$$\bar p= (ax+by+cz)(a_1x+b_1y+c_1z),\quad \bar q= (ax+by+cz)(a_2x+b_2y+c_2z).$$
Then we have
$$p(x,y)=\bar p(x,y,1)= (ax+by+c)(a_1x+b_1y+c_1),\quad \bar q= (ax+by+c)(a_2x+b_2y+c_2).$$
Thus condition $a')$ holds.

Now suppose that the condition b) of Theorem \ref{tm:1} holds, i.e.,
\item[b)] $\bar p(x, y, z) = f(\bar s, \bar t), \ \bar q(x, y, z) = g(\bar s, \bar t),$\\	
		 where $f$ and $g$ are bivariate homogeneous polynomials: $f,g\in\dot\Pi_2^2$ and \\
		\begin{equation*} \label{st2} \bar s = a_1x + b_1y + c_1z,\
		 \bar t = a_2x + b_2y + c_2z.
\end{equation*}

Thus we get

$p(x, y) =\bar p(x,y,1)= f(s, t), \ q(x, y)=\bar q(x,y,1) = g(s, t),$\\	
		 where $f$ and $g$ are bivariate homogeneous polynomials: $f,g\in\dot\Pi_2^2$ and \\
		\begin{equation*} s(x,y) =\bar s(x,y,1)= a_1x + b_1y + c_1,\
		 t(x,y) =\bar t(x,y,1)= a_2x + b_2y + c_2.
\end{equation*}
Thus the condition $b')$ holds if $a_1^2+b_1^2=0$ or $a_2^2+b_2^2=0.$ Otherwise
 the condition $c')$ holds.
\end{proof}	
\begin{theorem} \label{tm:40}
	Let $p= \ell_1 \ell_2, q = \ell_3\ell_4,$ where $\ell_i\in \dot\Pi_1^2,$ are linearly independent polynomials of degree two.  If the polynomial $r(x, y) = \alpha p(x, y)$ + $\beta q(x, y)$ is factorizable for two non-collinear vectors $(\alpha_1, \beta_1)$ and $(\alpha_2, \beta_2),$ with $\alpha_1\beta_1\alpha_2\beta_2\neq 0,$ then the polynomial $r(x,y)$ is factorizable for all pairs $(\alpha, \beta)\in\mathbb C^2.$

Moreover, at least one of the following conditions takes place:
	\begin{enumerate}
		\item[$(i')$] $p$ and $q$ have a common factor,
		\item[$(ii')$]The lines $\ell_1, \ell_2 , \ell_3, \ell_4$ are parallel,
\item[$(iii')$]   The lines $\ell_1, \ell_2 , \ell_3, \ell_4$ are coincident.
	\end{enumerate}
\end{theorem}
\begin{proof}
	We get, as in the previous proof, that  the hypothesis of Theorem \ref{tm:4} hold with the homogeneous trivariate polynomials $\bar{p}$, $\bar{q},$ which are associated with $p$ and $q.$ Then, as above we conclude that for any numbers $\alpha$ and $\beta$  the polynomial $r(x,y)$ is factorizable. It  remains to prove that at least one of conditions $(i'),(ii'),(iii')$ holds.

In the proof of Theorem \ref{tm:20} we verified that if $\bar p$ and $\bar q$ have a common factor then the same is true for the polynomials $p$ and $q.$

Thus the condition $(i')$ holds if $(i)$ holds.

Now suppose that the condition $(ii)$ of Theorem \ref{tm:1} holds, i.e.,
the lines $\bar\ell_1, \bar\ell_2 , \bar\ell_3, \bar\ell_4$ are coincident, where $\bar \ell_i$ is the homogeneous associate of $\ell_i.$

Suppose that the line $\ell_i$ is given by $a_ix + b_iy + c_i=0,\ i=1,\ldots,4.$

Suppose also that the four lines are coincident at $(x_0,y_0,z_0).$

If this point is at infinity, i.e., $z_0=0,$ then we have that $x_0^2+y_0^2\neq o$ and
 $a_ix_0 + b_iy_0=0,\ i=1,\ldots,4.$ This means that the vector $(x_0,y_0)$ is a normal vector for all lines $\ell_i,\ i=1,\dots,4.$ Therefore these lines are parallel and the condition $(ii')$ holds.

Next, suppose that the point of coincidence $(x_0,y_0,z_0)$ is a finite one, i.e., $z_0\neq 0.$ Then, in view of \eqref{aaa}, we readily get that the lines $\ell_i,\ i=1,\dots,4,$ are coincident
at $\left(\frac{x_0}{z_0},\frac{y_0}{z_0}\right).$ Therefore the condition $(iii')$ holds.
\end{proof}	
\subsection {A counterexample}

At the end let us modify the counterexample from Subsection \ref{ce} for the case of bivariate polynomials.

Consider the following two polynomials:
$$p(x,y,z)=x(x+1),\quad  q(x,y,z)=y(2x+y+1).$$

In the same way as in the homogeneous case we can show that for the three non-collinear vectors $(1,0); \ (0,1);$ and $(1,1)$ the respective linear combinations are factorizable, while the linear combination for $(2,1)$ is irreducible. Thus we conclude, in view of Theorem \ref{tm:20}, that all linear combinations except the first three are irreducible.

\vspace{3mm}


\noindent \emph{Anna Gharibyan} \vspace{2mm}

\noindent{Department of Informatics and Applied Mathematics\\
Yerevan State University\\
A. Manukyan St. 1\\
0025 Yerevan, Armenia}

\vspace{1mm}
\noindent E-mail: $<an.gharibyan@gmail.com>$

\end{document}